\documentclass[10pt]{amsart}

\usepackage{amssymb,amsthm,amsmath}
\usepackage[numbers,sort&compress]{natbib}
\usepackage{color}
\usepackage{graphicx}
\usepackage{tikz}

\title[ ]{Anderson localization for one-frequency quasi-periodic block operators with long-range interactions }

\thanks{Project supported by the National Natural Science Foundation of China (No.11421061).}
\author{Wenwen Jian}
\address[W. Jian]{School of Mathematical Sciences,
Fudan University,
Shanghai 200433,
P. R. China} \email{wwjian16@fudan.edu.cn}
\author{Yunfeng Shi}
\address[Y. Shi]{School of Mathematical Sciences,
Fudan University,
Shanghai 200433,
P. R. China} \email{yunfengshi13@fudan.edu.cn}
\address[Y. Shi]{Current address: School of Mathematical Sciences,
Peking University,
Beijing 100871,
P. R. China}
 \email{yunfengshi@math.pku.edu.cn}
\author{Xiaoping Yuan}
\address[X. Yuan]{School of Mathematical Sciences,
Fudan University,
Shanghai 200433,
P. R. China} \email{xpyuan@fudan.edu.cn}

\keywords{Anderson localization, quasi-periodic block operators, long-range interactions.}

\theoremstyle{plain}
\newtheorem{thm}{Theorem}[section]
 
 \newtheorem{lem}[thm]{Lemma}
 \newtheorem{prop}[thm]{Proposition}
 \theoremstyle{definition}
 \newtheorem{defn}[thm]{Definition}
 \theoremstyle{remark}
 \newtheorem{rem}[thm]{Remark}
 
 \numberwithin{equation}{section}
\begin{document}

% \newcommand{\N}{\mathbb{N}}
%%% ----------------------------------------------------------------------

\begin{abstract}
In this paper, we study the  quasi-periodic operators $H_{\epsilon,\omega}(x)$: $$(H_{\epsilon,\omega}(x)\vec{\psi})_n=\epsilon\sum_{k\in\mathbb{Z}}W_k\vec{\psi}_{n-k}+V(x+n\omega)\vec{\psi}_n,$$
where
$$\vec{\psi}=\{\vec{\psi}_n\}\in\ell^2(\mathbb{Z},\mathbb{C}^l),\ V(x)=\text{diag}\left(v_1(x),\cdots,v_l(x)\right)$$
with $v_i$ ($1\leq i \leq l$) being  real analytic functions on $\mathbb{T}=\mathbb{R}/\mathbb{Z}$
and $W_k$ ($k\in\mathbb{Z}$) being $l\times l$ matrices satisfying $\|W_k\|\leq C_0e^{-\rho|k|}$.
Using techniques developed by Bourgain and  Goldstein [\textit{{Ann. of Math.  152(3):835--879, 2000}}], we show that for  $|\epsilon|\leq \epsilon_{0}(V,\rho,l,C_0)$ ( depending only on $V,\rho, l, C_0$) and $x\in \mathbb{R}/\mathbb{Z}$,  there is some  full Lebesgue measure subset $\mathcal{F}$ of the Diophantine frequencies such that $H_{\epsilon,\omega}(x)$ exhibits Anderson localization if $\omega\in \mathcal{F}$.
\end{abstract}

%%% ----------------------------------------------------------------------
\maketitle
%%% ----------------------------------------------------------------------
%\tableofcontents
\section{Introduction and main result}
Quasi-periodic operators  have been widely studied  in both physics and mathematics literatures, and one of the most famous and typical operators of such type may be the almost Mathieu operator (AMO for short):
\begin{equation*}\label{amo}
(H_{\lambda,\omega,x}u)_n=u_{n+1}+u_{n-1}+2\lambda\cos2\pi(x+n\omega)u_n,
\end{equation*}
where $u=\{u_n\}\in \ell^2(\mathbb{Z},\mathbb{C}),x\in\mathbb{T}$ and $\omega\in\mathbb{R}\setminus\mathbb{Q}$.
In recent years, more and more  research efforts have focused on  the nature of the spectrum and the behaviour of the eigenfunctions, particularly on phenomenon of Anderson localization (AL for short) which means the operator has pure point spectrum with exponentially decaying eigenfunctions. The methods for establishing the AL for a quasi-periodic operator  include mainly the perturbative one and the non-perturbative one. The KAM technique  is a typical perturbative method, which relies heavily on intricate multi-step procedures, eigenvalue (eigenfunction) parametrization, and perturbation arguments \cite{FSWC,SJ,DSF,EA,EC,ST}. Thus the perturbation  may depend on the Diophantine condition. However, the non-perturbative method treats the Green's function directly and only finite scales are involved. As a result, in many cases,  the smallness (largeness) of the perturbation is independent of the Diophantine condition (this is called a non-perturbative AL). For an elegant and more complete exposition of (non) perturbative results (methods), we refer the reader to \cite{JI} by Jitomirskaya.
Let us give a more exact introduction of the non-perturbative AL results.
In 1999, Jitomirskaya \cite{JA} showed that the AMO $H_{\lambda,\omega,x}$ exhibits AL for almost every  $x\in\mathbb{T}$ if $\omega$ is a Diophantine frequency  and $|\lambda|>1$, here $\omega$ is a Diophantine frequency means there is $t>0$ such that $\omega\in \text{DC}_t$ with
$$\mathrm{DC}_t
=\left\{\omega\in\mathbb{R}: ||k\omega||_{\mathbb{T}}\geq \frac{t}{|k|^2},\  \forall k\in\mathbb{Z}\setminus\{0\}\right\}\footnote{Where $ ||x||_{\mathbb{T}}:=\min\limits_{k\in\mathbb{Z}}|x-k|$. It is well-known that the Lebesgue measure of $\mathrm{DC}_t$ is $1-\mathcal{O}(t)$. Our definition here is a little different from that in \cite{JA,BGA}. However, it is not essential.}.$$
Subsequently,  Bourgain and Goldstein \cite{BGA} proved that for a non-constant real analytic potential $v$ on $\mathbb{T}$, the general one-frequency quasi-periodic Schr\"odinger operators which are given by
 $$(H_{\lambda,\omega,x}u)_n=u_{n+1}+u_{n-1}+\lambda v(x+n\omega)u_n,$$
satisfy  AL  with  $\omega$ being in a full Lebesgue measure subset of $\mathrm{DC}_t$ and $|\lambda|\geq \lambda_0(v)\gg1$ (independent of $\mathrm{DC}_t$ ). In Chapter 11 of Bourgain's monograph \cite{BB}, he extended their  result  of \cite{BGA} to long-range operators (actually a sketch of the proof):
\begin{equation*}\label{long}
(H_{\epsilon,\omega,x}u)_n=\epsilon\sum_{k\in\mathbb{Z}}w_ku_{n-k}+v(x+n\omega)u_n,
\end{equation*}
where $\{w_k\}_{k\in\mathbb{Z}}$ are the Fourier coefficients of some real analytic function $w$ on $\mathbb{T}$ and $|\epsilon|\leq \epsilon_0(w,v)\ll1$.
We then turn to the block operators case. In \cite{BJG}, Bourgain and Jitomirskaya  extended the result of \cite{BGA} to the band Schr\"odinger operators:
\begin{equation*}\label{band}
 (H_{\lambda,\omega}(x)\vec{\psi})_n:=\vec{\psi}_{n-1}+\vec{\psi}_{n+1}+(\lambda V(x+n\omega)+W_0)\vec{\psi}_n,
\end{equation*}
 with
  $$W_{0}= \left(
           \begin{array}{cccc}
             \ddots& \ddots & 0&\ddots \\
             1& 0 &1 &0 \\
             0 & 1  & 0&1 \\
             \ddots&0& \ddots & \ddots\\
           \end{array}
         \right)_{l\times l}$$
 and $V(x)=\text{diag}\left(v_1(x),\cdots,v_l(x)\right),\vec{\psi}=\{\vec{\psi}_n\}\in\ell^{2}(\mathbb{Z},\mathbb{C}^l)$. In a recent paper  by Klein \cite{KJF}, he studied the  quasi-periodic block Jacobi operators:
 \begin{equation*}\label{jacobi}
 (H_{\lambda,\omega}(x)\vec{\psi})_n:=-(\triangle_W(x)\vec{\psi})_n +\lambda V(x+n\omega)\vec{\psi}_n
 \end{equation*}
 with the ``weighted" Laplacian
 $$(\triangle_W(x)\vec{\psi})_n:=W(x+(n+1)\omega)\vec{\psi}_{n+1}+W^{\top}(x+n\omega)\vec{\psi}_{n-1}+R(x+n\omega)\vec{\psi}_{n}.$$
 Klein proved a non-perturbative AL and generalized the result of \cite{BJG}. For recent AL results,  we  refer the reader to  \cite{B2007,BGS2002,GSV,GS2008,GS2011,AJA,AYZD,JLA}.

In this paper, we study the one-frequency quasi-periodic block operators with long-range interactions:
\begin{equation}\label{op}
(H_{\epsilon,\omega}(x)\vec{\psi})_n:=\epsilon\sum_{k\in\mathbb{Z}}W_k\vec{\psi}_{n-k}+V(x+n\omega)\vec{\psi}_n,
\end{equation}
where
$$\vec{\psi}=\{\vec{\psi}_n\}\in\ell^2(\mathbb{Z},\mathbb{C}^l),\ V(x)=\text{diag}\left(v_1(x),\cdots,v_l(x)\right)$$
with $v_i$ ($1\leq i \leq l$) being  real analytic functions on $\mathbb{T}$
and $W_k$ ($k\in\mathbb{Z}$) being exponential decay $l \times l$ matrices satisfying $W_{-k}=W_k^*$ ($W_k^*$ denotes the complex conjugate of $W_k$).
In general, we call $x\in \mathbb{T}$ the phase, $\omega\in\mathbb{R}\setminus\mathbb{Q}$ the frequency, $\epsilon\in\mathbb{R}$ the perturbation and $V$ the potential. It is well-known that  every $H_{\epsilon,\omega}(x)$  is a bounded self-adjoint operator on the Hilbert space $\ell^2(\mathbb{Z},\mathbb{C}^l)$.
This kind of operators was studied in some papers before, such as in \cite{GSC,GC}.

The purpose of the present work is to show the operator $H_{\epsilon, \omega}(x)$ defined in (\ref{op}) exhibits non-perturbative AL. This generalizes a result of  Bourgain \cite{BB} as well as a result of  Klein \cite{KJF}.
%We  give the precise description of the main result of this paper.
 %Let $C_{r}^{\omega}(\mathbb{T}, \mathrm{Mat}_{l}(\mathbb{C}))$  (resp. $C_{r}^{\omega}(\mathbb{T}, \mathrm{Mat}_{l}(\mathbb{R}))$) be the set of all analytic mappings from the strip $\{x\in\mathbb{C}/\mathbb{Z}: |\Im x|\leq r\}$ to the space $\mathrm{Mat}_{l}(\mathbb{C})$ (resp. $\mathrm{Mat}_{l}(\mathbb{R})$) of all $l\times l$ complex (resp. real) matrices.
 % In the following, we use the standard  Euclidean norm for a vector and the operator norm for a matrix.  Our main theorem is the following.
More precisely, we have
\begin{thm}\label{mthm}
Let $H_{\epsilon, \omega}(x)$ be given by (\ref{op}) with $||W_k||\leq C_0 e^{-\rho|k|}$  and $v_i$ ($1\leq i\leq l$) be nonconstant real analytic functions on $\mathbb{T}$, where the norm $\|\cdot\|$ is the standard matrix norm. Then there exists $\epsilon_0=\epsilon_0(V,\rho,l,C_0)>0$ (depending only on $V,\rho,l,C_0$) such that for $|\epsilon|\leq |\epsilon_0|$, $x\in\mathbb{T}$, there is some zero Lebesgue measure set $\mathcal{R}$ so that for $\omega\in \mathrm{DC}_t\setminus\mathcal{R},$  $H_{\epsilon,\omega}(x)$ shows the Anderson localization.
\end{thm}
%\begin{rem}\label{rem1}
%From Sch'nol-Simon theorem (see [15] for details), to prove AL, it suffices to show that every extended state (defied in \S6) decays exponentially. Generally, there are two steps. First, one establishes large deviation estimates (LDT for short) for the restricted Green's function by applying subharmonic function theory, ergodic properties and  semi-algebraic set properties. Second, one eliminates the energy E and establishes localization.

The  proof  of our main theorem employs  techniques developed by  Bourgain and  Goldstein  in \cite{BGA}. We also use some tools in \cite{BB,BJI},  and some convenient notations of Klein in \cite{KJF}.

The main difficulty  here is to establish large deviation theorem (LDT for short) for the restricted Green's function $G_N(x;E)$ (see \S5) and this leads to exploring efficiently upper bounds on  minors of the block matrix $H_N(x)-EI_N$ (see \S 3) as well as a lower bound on $\int_{\mathbb{T}}\log|\text{det}[H_N(x)-EI_N]|\mathrm{d}x$. Since our block operator $H_{\epsilon,\omega}(x)$ (see (\ref{op})) is with a long-range perturbation, it is much more complicated and skillful to obtain such upper and lower bounds.
%\end{rem}

%It is well-known that to establish non-perturbative  Anderson localization, it is enough to prove the exponential decay of the eigenfunctions. Generally, there are two steps. First, one establishes LDT for the restricted Green's function by applying subharmonic function theory and ergodic properties. Second, one eliminates the energy $E$ and establishes localization.

The structure of the paper is as follows. In \S2, we introduce some notations and basic concepts. In \S 3, we prove uniformly upper bounds on the minors of the Dirichlet matrix. In \S4, we obtain a lower bound on the average of the Dirichelet determinant on torus. The Green's function estimates are established in \S 5. In \S 6, we finish the proof of our main theorem. We include some useful  lemmata  in Appendix A.

\section{Some basic concepts and notations}

\subsection{Some notations}
We use convenient notations introduced by Klein in \cite{KJF}. Let $\mathrm{Mat}_m(\mathbb{C})$ be the set of all $m\times m$ complex matrices. Given a block matrix $M$, we use roman letters for the indices of its block-matrix entries, and Greek letters for the indices of its scalar entries. More precisely, we write $M=(M_{\gamma,\gamma'})_{1\leq \gamma,\gamma'\leq Nl}\in \mathrm{Mat}_{Nl}(\mathbb{C})$ which can be identified with a block matrix $M=(\underline{M}_{n,n'})_{1\leq n,n'\leq N} $ (i.e. $\underline{M}_{n,n'}\in \mathrm{Mat}_l(\mathbb{C})$ for any $1\leq n,n'\leq N$).
Moreover, given $\gamma\in[1,Nl]$, there is a unique $1\leq n(\gamma)\leq N$ such that $\gamma=l\cdot (n(\gamma)-1)+r$ with $1\leq r\leq l$. Thus any scalar $M_{\gamma,\gamma'}$ belongs to the block $\underline{M}_{n(\gamma),n(\gamma')}$.

Given any interval  $[a,b]=\mathcal{N}\subset\mathbb{Z}$ with length $|\mathcal{N}|=b-a+1$ and any infinite $l\times l$-block matrix $M=(\underline{M}_{n,n'})_{n,n'\in\mathbb{Z}}$ (i.e. $\underline{M}_{n,n'}\in \mathrm{Mat}_l(\mathbb{C})$ for all $n,n'\in\mathbb{Z}$), we denote by  $M_{\mathcal{N}}=(\underline{M}_{n,n'})_{n,n'\in\mathcal{N}}$. More generally, one can  define $M_{\mathcal{N}_1,\mathcal{N}_2}$ by restricting  $n\in\mathcal{N}_1,n'\in\mathcal{N}_2$. Especially, we write $M_N=M_{\mathcal{N}}$ if $\mathcal{N}=[1,N]$. Finally, by $I$, we mean the block identity matrix, that is $I=\text{diag}(I_{n})_{n\in\mathbb{Z}}$ with $I_{n}$ being  $l\times l$ identity matrix.

We define for $\rho\geq 0$ the strip $\Delta_\rho=\{z\in\mathbb{C}/\mathbb{Z}: |\Im z|\leq \rho\}$.
For any continuous mapping $f$ from $\Delta_\rho$ to some Banach space $(\mathcal{B},\|\cdot\|)$, we define $\|f\|_\rho=\sup\limits_{z\in\Delta_\rho}\|f(z)\|$.
For any measurable set $\mathcal{A}\subset\mathbb{R}$, we denote by $\mathrm{Leb}(\mathcal{A})$ its Lebesgue measure. If a constant $C$ depends only on functions $v_i(x)$ ($1\leq i\leq l$), we write $C=C(V)$ with $V(x)=\mathrm{diag}(v_1(x),\cdots,v_l(x))$. We also use Euclidean norm for a vector and the standard operator norm for a matrix.

Note that every real analytic function $f$ on $\mathbb{T}$ can be analytically extended to the strip $\Delta_{c(f)}$ with $c(f)>0$ depending only on $f$. Thus without loss of generality,  we assume each $v_i(x)$ ($1\leq i\leq l$) is analytic on $\Delta_\rho$. For simplicity, we also assume the perturbation $\epsilon\geq 0$ and $\|W_k\|\leq e^{-\rho|k|}$ (i.e., $C_0=1$).

\subsection{Harmonic measure} For reader's convenience, we introduce the basic properties of the harmonic measure which will be used in \S 5. The materials in this subsection are from \cite{GMB}.
Write $\mathbb{H}=\{z\in\mathbb{C}:\Im z>0\}$ for the upper half-plane and $\partial \mathbb{H}=\mathbb{R}$ for its boundary.
If $U\subset\mathbb{R}$ is measurable,  the harmonic measure of $U$ at $z=x+iy\in \mathbb{H}$ is
\begin{equation*}
  \mu(z,U,\mathbb{H})=\int_U\frac{y}{(t-x)^2+y^2}\frac{\text{d}t}{\pi}.
\end{equation*}
We note that
\begin{itemize}
\item If $U=(a,b)$, then
\begin{equation}\label{mu1}
  \mu(z,(a,b),\mathbb{H})=\frac{1}{\pi}\arg\left(\frac{z-b}{z-a}\right).
\end{equation}
\item For any bounded Borel function $f$ on $\mathbb{R}$, we have for $x+iy\in\mathbb{H}$
\begin{equation}\label{hmi}
 \int_{\mathbb{R}}f(t)\mathrm{d}\mu(x+iy,t,\mathbb{H})=\int_{\mathbb{R}}f(t+x)\mathrm{d}\mu(iy,t,\mathbb{H}).
\end{equation}

\end{itemize}

Let $\Omega$ be a simply connected domain in the extended plane $\mathbb{C}^{\ast}=\mathbb{C}\cup\{\infty\}$ with its boundary $\partial\Omega$ being a Jordan curve in $\mathbb{C}^{\ast}$.
If $\phi$ is a  conformal mapping from  $\mathbb{H}$ onto $\Omega$, then by Carath\'eodory's theorem, $\phi$ has a continuous extension (again denoted by $\phi$) to $\overline{\Omega}$ and this extension is a continuous bijective mapping from $\overline{\mathbb{H}}$ to $\overline{\Omega}$. Then for any Borel set $U\subset\partial\Omega$, we can define the harmonic measure of $U$ relative to $\Omega$ at $z\in \Omega$ by
\begin{equation}\label{mu3}
    \mu(z,U,\Omega)=\mu(\phi^{-1}(z),\phi^{-1}(U),\mathbb{H}).
\end{equation}
\begin{rem}
This definition is independent of  the choices of conformal mappings.
\end{rem}

\section{Uniformly upper bounds on minors of  the Dirichlet matrix }
In this section, we will prove  uniformly upper bounds on minors of the Dirichlet matrix and we use tools in \cite{BJI} (see also Chapter 11 of \cite{BB}).

Let $\mathcal{N}\subset\mathbb{Z}$ be an interval and $H_{\mathcal{N}}(x)$ be the  restriction of $H_{\epsilon,\omega}(x)$  on $\ell^2(\mathcal{N},\mathbb{C}^l)$. We fix $\epsilon$ and restrict $E$ in a compact interval $\mathcal{A}\subset \mathbb{R}$.
Then $H_N(x,E):=H_N(x)-EI_N$ can  be represented by a $Nl\times Nl$ matrix with complex entries, which we denote by $H_{N,(\alpha,\alpha')}(x,E)$, where $1\leq \alpha,\alpha'\leq Nl$.
 We let $\mu_{N,(\alpha,\alpha')}(x,E)$ be  the $(\alpha,\alpha')$-minor of the Dirichlet matrix $H_N(x,E)$. That is,
$$\mu_{N,(\alpha,\alpha')}(x,E):= \det[(H_N(x,E))_{\neg\alpha,\neg\alpha'}],$$
where for any $1\leq \gamma\leq Nl$, $\neg\gamma=\{1,\cdots,Nl\}\setminus\{\gamma\}$ is arranged in natural order.

The following lemma gives an expression for a minor of a matrix.

\begin{lem}[Lemma 10 in \cite{KJF}]\label{mu}
Let $G=(G_{\gamma,\gamma'})\in\emph{Mat}_{m}(\mathbb{C})$ and $\alpha,\alpha'\in\{1,2,\cdots,m\}$ with $\alpha\neq\alpha'$. Then
\begin{equation*}
  \det[G_{\neg\alpha,\neg\alpha'}]=\sum_{\Gamma}A_{\Gamma}\det [G_{\Gamma^{c}}]\prod_{i=1}^{s-1}G_{\gamma_{i},\gamma_{i+1}},
\end{equation*}
where the sum is taken over all ordered subsets $\Gamma=(\gamma_1,\cdots,\gamma_s)$ of $\{1,2,\cdots,m\}$ with $\gamma_1=\alpha$, $\gamma_s=\alpha'$, and  $A_{\Gamma}\in\{-1,1\}$, $\Gamma^c=\{1,\cdots,m\} \setminus \Gamma$ is arranged in natural order.
\end{lem}

\begin{lem}\label{lb}
Assume $b=\sum\limits_{j=1}^{s-1}|n(\gamma_{j+1})-n(\gamma_j)|$. Then $s\leq (b+1)l$.
\end{lem}
\begin{proof}
We let $n(\gamma_1)=n(\gamma_2)=\cdots=n(\gamma_{l_1})=n_1$, $n(\gamma_{l_1+1})=n(\gamma_{l_1+2})=\cdots=n(\gamma_{l_2})=n_2,\cdots, n(\gamma_{l_{m-1}+1})=n(\gamma_{l_{m-1}+2})=\cdots=n(\gamma_{l_m})=n_m$,
where $l_m=s$ and $n_i\neq n_j$ for $i\neq j$.
Since each block has order  $l$, we must have $l_1\leq l,\ l_j-l_{j-1}\leq l$ ($2\le j\leq m$). As a result, $s\leq ml$. Finally,
$$b= \sum\limits_{j=1}^{m-1}|n(\gamma_{l_j})-n(\gamma_{l_j+1})|= \sum\limits_{j=1}^{m-1}|n_j-n_{j+1}|\geq m-1\geq \frac{s}{l}-1.$$
The lemma follows.
\end{proof}

\begin{lem}[Lemma 11.29  in \cite{BB}]\label{bous}
Let $\Gamma=(\gamma_1,\cdots,\gamma_s)$ be an ordered subset of $\{1,2,\cdots,Nl\}$ and $v$ be a non-constant real analytic function on $\mathbb{T}$. Then there is some $\delta=\delta(v)>0$ such that for $0<\epsilon\ll1$, $s>\epsilon^{\delta}Nl$ and $E\in\mathbb{R}$,
\begin{equation*}
 \inf\limits_{x\in\mathbb{T}}\sum_{k\in\Gamma} \log(|v(x+k\omega)-E|+\epsilon)\geq \frac{3}{4}s\log\epsilon.
\end{equation*}
\end{lem}

We can now state our main result of this section.

\begin{prop}\label{upperbound} Let $\omega\in\mathrm{DC}_t$.
There are $\sigma=\sigma(V)>0$ (depending only on $V$) and $\epsilon_0=\epsilon_0(l,\rho)>0$ (depending only on $l,\rho$) such that if $0<\epsilon\leq \epsilon_0$, then for all $x\in \mathbb{T}$ and $N\gg1$ (depending on $V,\epsilon,l,\rho,t$), we have
\begin{equation*}
   |\mu_{N,(\alpha,\alpha')}(x,E)|\leq e^{N(\sum\limits_{j=1}^l\int_{\mathbb{T}}\log|v_j(x)-E|\mathrm{d}x+l\epsilon^{\sigma})} e^{-(\rho-\epsilon^{\sigma})|n(\alpha)-n(\alpha')|},
\end{equation*}
where $1\leq\alpha,\alpha'\leq Nl$.
\end{prop}

\begin{proof}
The proof of case $\alpha=\alpha'$ is trivial. Thus in the following, we only consider $\alpha\neq\alpha'$.
From Lemma \ref{mu},
\begin{equation*}
  \mu_{N,(\alpha,\alpha')}(x,E)=\sum_{s}\sum_{\Gamma:|\Gamma|= s}
                \pm\det[(H_N(x,E))_{\Gamma^c}]\epsilon^{s-1}\prod_{j=1}^{s-1}H_N(x,E)_{\gamma_{j},\gamma_{j+1}},
\end{equation*}
where $\Gamma=(\gamma_1,\gamma_2,\cdots,\gamma_s)$ is a path in $[1,Nl]$ with $\gamma_1=\alpha, \gamma_s=\alpha'$.
Recalling the properties of $\{W_k\}$, one has
\begin{align*}
  |\mu_{N,(\alpha,\alpha')}(x,E)|
  \leq &\ \sum_{s}\sum_{\Gamma:|\Gamma|= s} |\det[(H_N(x,E))_{\Gamma^c}]|\epsilon^{s-1}\prod_{j=1}^{s-1}\|W_{n(\gamma_{j})-n(\gamma_{j+1})}\|\\
  \leq &  \sum_{s}\sum_{\Gamma:|\Gamma|= s}
          \epsilon^{s-1}e^{-\rho\sum\limits_{j=1}^{s-1}|n(\gamma_{j+1})-n(\gamma_{j})|} |\det[(H_N(x,E))_{\Gamma^c}]| \\
  \leq & \sum_{b\geq|n(\alpha)-n(\alpha')|}\sum_{s}(2l)^{s-1} \binom{b+s-1}{s-1}\epsilon^{s-1}e^{-\rho b}\\
               & \ \ \times\max_{\Gamma(s,b)}|\det[(H_N(x,E))_{\Gamma^c(s,b)}]|,
\end{align*}
where $b=\sum\limits_{j=1}^{s-1}|n(\gamma_{j+1})-n(\gamma_{j})|$, and a $(s,b)$-path $\Gamma(s,b)$ means the path $\Gamma$ with the restrictions $s=|\Gamma|$ and $b=\sum\limits_{j=1}^{s-1}|n(\gamma_{j+1})-n(\gamma_{j})|$. In the last inequality, we use the fact that  there are at most $(2l)^{s-1} \binom{b+s-1}{s-1}$ many $(s,b)$-paths.

Thus it needs to give the upper bound on $|\det[(H_N(x,E))_{\Gamma^c}]|$ first. By Hadamard's inequality,
\begin{align*}
  |\det[(H_N(x,E))_{\Gamma^{c}}]|
  %\leq & \prod_{\alpha\in \gamma^c}\legt(\sum_{\alpha '\in \gamma^c}\right) \\
  \leq & \prod_{\beta\in \Gamma^c}\left[|v_{j(\beta)}(x+n(\beta)\omega)-E|+C\epsilon\right],
\end{align*}
where $\beta=(n(\beta)-1)l+j(\beta)$ with  $1\leq\beta\leq Nl $, $0<j(\beta)\leq l$ and $C>0$ is a constant depending only on $\rho$.
Hence
\begin{align}
\nonumber \log |\det[(H_N(x,E))_{\Gamma^{c}}]|= \ & \sum_{j=1}^l\sum_{n=1}^{N}\log\left[|v_{j}(x+n\omega)-E|+C\epsilon\right]\\
  \nonumber & \ - \sum_{\beta\in \Gamma}\log\left[|v_{j(\beta)}(x+n(\beta)\omega)-E|+C\epsilon\right]\\
  \label{logdet} :=\ & (\textrm{I})-(\textrm{II}).
\end{align}
 From Lemmata \ref{1} and \ref{dk}, there is some $\delta=\delta(V)>0$ such that for $0<\epsilon\leq \epsilon_0(\rho)$ and $N\gg1$ (depending on $V,\epsilon,t,\delta$),
\begin{align}\label{I}
 (\textrm{I})
  \leq \ & N\sum\limits_{j=1}^l \int_{\mathbb{T}} \log|v_j(x)-E|\text{d}x + Nl\epsilon^{\delta}.
\end{align}
For (\textrm{II}), using Lemma \ref{bous}, there is some $\delta_1=\delta_1(V)>0$ such that for $s\geq 1$,
\begin{equation}\label{II}
(\textrm{II})>s\log\epsilon,
\end{equation}
and for $s\geq \epsilon^{\delta_1}Nl$,
\begin{align}\label{IIbetter}
  (\textrm{II})\geq\frac{3}{4}s\log\epsilon,
\end{align}
where $s=|\Gamma|$. From (\ref{logdet}), (\ref{I}), (\ref{II}) and (\ref{IIbetter}), we have
\begin{equation}\label{2}
  \log|\det[(H_N(x,E))_{\Gamma^{c}}]|\leq N\sum_{j=1}^l\int_{\mathbb{T}}\log|v_j(x)-E|\text{d}x+Nl\epsilon^{\delta}+s\log\frac{1}{\epsilon},
\end{equation}
and for $s>\epsilon^{\delta_1}Nl$,
\begin{equation}\label{2better}
  \log|\det[(H_N(x,E))_{\Gamma^{c}}]|\leq N\sum_{j=1}^l\int_{\mathbb{T}}\log|v_j(x)-E|\text{d}x+Nl\epsilon^{\delta}+\frac{3}{4}s\log\frac{1}{\epsilon}.
\end{equation}

We then can finish our proof as follows. Let $$d_0=\sum\limits_{j=1}^l\int_{\mathbb{T}}\log|v_j(x)-E|\text{d}x+l\epsilon^{\delta}.$$
Recalling (\ref{2}), (\ref{2better}) and Lemma \ref{lb}, we obtain
\begin{align}\label{mu0}
 \nonumber |\mu_{N,(\alpha,\alpha')}(x,E)|\leq\ & \epsilon^{-1} e^{Nd_0} \sum_{b\geq|n(\alpha)-n(\alpha')|}\sum_{\substack{s\leq (b+1)l\\s\leq \epsilon^{\delta_1}Nl}}(2l)^{s-1} \binom{b+s-1}{s-1}e^{-\rho b}\\
\nonumber  & + \epsilon^{-\frac{3}{4}}  e^{Nd_0}   \sum_{b\geq|n(\alpha)-n(\alpha')|}\sum_{\substack{s\leq (b+1)l\\s>\epsilon^{\delta_1}Nl}}
                   (2l\epsilon^{\frac{1}{4}})^{s-1} \binom{b+s-1}{s-1}e^{-\rho b}\\
        := & \ (\textrm{III})+(\textrm{IV}).
\end{align}
Regarding (IV), we have for $N\gg1$ and $2l(2l+1)\epsilon^{\frac{1}{4}}<\rho $,
\begin{align}\label{4}
 \nonumber (\textrm{IV})\leq &\  \epsilon^{-\frac{3}{4}} e^{Nd_0} \sum_{b\geq|n(\alpha)-n(\alpha')|} e^{-\rho b}
                                  \sum_{\substack{s\leq (b+1)l\\s>\epsilon^{\delta_1}{Nl}}}(2l\epsilon^{\frac{1}{4}})^{s-1} \binom{b(2l+1)}{s-1}\\
 \nonumber \leq & \  \epsilon^{-\frac{3}{4}} e^{Nd_0} \sum_{b\geq|n(\alpha)-n(\alpha')|} e^{-\rho b}
                     (1+2l\epsilon^{\frac{1}{4}})^{b(2l+1)}\\
 \nonumber \leq &\   \epsilon^{-\frac{3}{4}} e^{Nd_0}\sum_{b\geq|n(\alpha)-n(\alpha')|}
                      e^{-\rho b}e^{2l(2l+1)\epsilon^{\frac{1}{4}}b}\\
 \leq& \ C \epsilon^{-\frac{3}{4}} e^{Nd_0}e^{(-\rho+2l(2l+1)\epsilon^{\frac{1}{4}}) |n(\alpha)-n(\alpha')|}.
\end{align}
For $(\textrm{III})$, we distinguish the cases $|n(\alpha)-n(\alpha')|>\epsilon^{\frac{\delta_1}{2}}Nl$ and $|n(\alpha)-n(\alpha')|\leq\epsilon^{\frac{\delta_1}{2}}Nl$.
 If $|n(\alpha)-n(\alpha')|>\epsilon^{\frac{\delta_1}{2}}Nl$, we obtain
\begin{align*}
  (\textrm{III})\leq &\ \epsilon^{-1} (2l)^{\epsilon^{\delta_1}Nl}  e^{Nd_0}
               \sum_{b\geq|n(\alpha)-n(\alpha')|}e^{-\rho b} \sum_{\substack{s\leq (b+1)l\\s\leq\epsilon^{\delta_1}Nl}} \binom{b+s-1}{s-1}\\
  \leq & \  \epsilon^{-1+\delta_1} Nl(2l)^{\epsilon^{{\delta_1}}Nl}  e^{Nd_0}
         \sum_{b\geq|n(\alpha)-n(\alpha')|}e^{-\rho b}
            \binom{b(2l+1)}{\epsilon^{\frac{\delta_1}{2}}b}.
\end{align*}
By Stirling formula, it appears $\binom{m}{rm}\leq Cme^{\phi(r)m}$, where
 \begin{equation*}
 \phi(r)=-r\log r-(1-r)\log(1-r),\  0<r<1.
 \end{equation*}
Therefore, when $N$ is large enough (s.t. $\frac{\log b}{b}<\epsilon$) and $-\rho+(2l+1)\phi((2l+1)^{-1}\epsilon^{\frac{\delta_1}{2}})+\epsilon<0$, we have
\begin{align}\label{3}
  \nonumber (\textrm{III})\leq  &\   C\epsilon^{-1+\delta_1} (2l+1)Nl(2l)^{\epsilon^{{\delta_1}}Nl}  e^{Nd_0}
            \sum_{b\geq|n(\alpha)-n(\alpha')|}be^{-\rho b}e^{(2l+1)\phi((2l+1)^{-1}\epsilon^{\frac{\delta_1}{2}})b}\\
   \leq &\ C\epsilon^{-1+\delta_1} Nl^2(2l)^{\epsilon^{\delta_1}Nl}  e^{Nd_0}
                            e^{(-\rho+\epsilon+(2l+1)\phi((2l+1)^{-1}\epsilon^{\frac{\delta_1}{2}})  )|n(\alpha)-n(\alpha')|}.
\end{align}
\item If $|n(\alpha)-n(\alpha')|\leq \epsilon^{\frac{\delta_1}{2}}Nl$, we write
\begin{align*}
  (\textrm{III})= &\sum_{b>\epsilon^{\frac{\delta_1}{2}}Nl} \sum_{\substack{s\leq (b+1)l\\s\leq\epsilon^{\delta_1}Nl}}\cdots
                     \ +\sum_{\substack{b\geq|n(\alpha)-n(\alpha')|\\b\leq\epsilon^{\frac{\delta_1}{2}}Nl}}
                        \sum_{\substack{s\leq (b+1)l\\s\leq\epsilon^{\delta_1}Nl}}  \cdots\\
                :=&\ (\textrm{III}_1)+(\textrm{III}_2).
\end{align*}
Similarly to the proof of (\ref{3}), one has
\begin{align}
  \nonumber(\textrm{III}_1)\leq &\ \epsilon^{-1+\delta_1} Nl (2l)^{\epsilon^{\delta_1}Nl}  e^{Nd_0}
                           \sum_{b>\epsilon^{\frac{\delta_1}{2}}Nl}e^{-\rho b}
                                \binom{b(2l+1)}{\epsilon^{\frac{\delta_1}{2}}b} \\
  \label{thr1}\leq &\  C\epsilon^{-1+\delta_1} Nl^2(2l)^{\epsilon^{\delta_1}Nl}  e^{Nd_0}
                            e^{(-\rho+\epsilon+(2l+1)\phi((2l+1)^{-1}\epsilon^{\frac{\delta_1}{2}})  )|n(\alpha)-n(\alpha')|}.
\end{align}
Note also that
\begin{align}
 \nonumber (\textrm{III}_2)\leq & \    \epsilon^{-1}  e^{Nd_0}
                           \sum_{\substack{b\geq|n(\alpha)-n(\alpha')|\\b\leq\epsilon^{\frac{\delta_1}{2}}Nl}}e^{-\rho b} (1+2l)^{b(2l+1)}\\
 \label{thr2} \leq &\  C\epsilon^{-1+\frac{\delta_1}{2}}Nl (3l)^{\epsilon^{\frac{\delta_1}{2}}Nl(2l+1)} e^{Nd_0} e^{-\rho |n(\alpha)-n(\alpha')|}.
\end{align}
Putting all above estimates (\ref{mu0})-(\ref{thr2}) together, we obtain that there are $\sigma=\sigma(V)>0,\ \epsilon_0=\epsilon_0(l,\rho)>0$ such that if $0<\epsilon\leq\epsilon_0$, for $N\gg1$ (depending on $V,\epsilon,l,\rho,t$),
$$|\mu_{N,(\alpha,\alpha')}(x,E)|\leq e^{Nd_0+\epsilon^\sigma Nl} e^{-(\rho-\epsilon^{\sigma})|n(\alpha)-n(\alpha')|},$$
which completes the proof.
\end{proof}

\section{A lower bound on average of the Dirichlet determinant on torus}

In this section, we will give a lower bound on the average of the determinant on  torus and the key is to estimate a subharmonic function. We use the complexification idea of   Sorets and Spencer in \cite{SSC}. The  technical tools  employed here are the harmonic measure estimates of Bourgain and Goldstein  in \cite{BGA}  together with the  quantitative Sorets-Spencer result of  Duarte and Klein in \cite{DKC}. We assume $E$ belongs to an interval $\mathcal{A}\subset\mathbb{R}$ with  $\mathrm{Leb}(\mathcal{A})\leq C(V,\rho)$.

We begin with a  lemma (i.e. quantitative Sorets-Spencer result).
\begin{lem}\label{v-e}
For any $0<\xi<\rho$, there are $c=c(V,\rho)>0, \Sigma=\Sigma(V,\rho)>0$ such that
\begin{equation}
\min_{E\in\mathbb{R}}\max_{\frac{\xi}{2}<y<\xi}\min_{1\leq j\leq l}\min_{x\in\mathbb{R}}|v_j(x\pm iy)-E|>c \xi^{\Sigma}.
\end{equation}
\end{lem}

\begin{proof}
This needs a small modification of the proof of Proposition 4.3 in \cite{DKC}. More precisely, we denote
\begin{equation*}
  \Sigma=\sum_{j=1}^l \sup_{E\in\mathbb{R}}\Sigma_j(E),
\end{equation*}
where $\Sigma_j(E)$ is the number (counting multiplicities) of zeros of function $v_j(z)-E$ on the strip $\Delta_\rho$. From Proposition 4.2 of \cite{DKC},  $\Sigma<\infty$. We let  $z_{j,1},\cdots,z_{j,k_j}$ be distinct zeros of $v_j(z)-E$ on
 $\Delta_\rho$ with multiplicities $n_{j,1},\cdots,n_{j,k_j}$. Then for any $E\in\mathbb{R}$,
$$\Sigma(E)=\sum_{j=1}^{l}\sum_{m=1}^{k_j}n_{j,m}\leq \Sigma<\infty.$$
Write
\begin{equation*}
  v_j(z)-E=g_{j,E}(z)\prod_{m=1}^{k_j}(z-z_{j,m})^{n_{j,m}},
\end{equation*}
where $g_{j,E}$ is the zero-free part of $v_j-E$ on $\Delta_\rho$. It was proved in \cite{DKC} (see Proposition 4.2. of \cite{DKC}) that
\begin{equation*}
  q=\min_{1\leq j\leq l}\inf_{E\in\mathbb{R}}\inf_{z\in\Delta_\rho}|g_{j,E}(z)|>0.
\end{equation*}

We now define  $\Omega_\xi:=\{z\in\mathbb{C}/\mathbb{Z}:\frac{\xi}{2}<|\Im z|<\xi\}$ and divide the region $\Omega_\xi$ into $8\Sigma+4$ parallel strips (along the real axis) such that every strip has the width $\frac{\xi}{8\Sigma+4}$. Then there are at least two symmetric (on the real axis) strips containing no zero in their interiors.
We denote $\Omega_0$ one of such strips with $\Im z>0$ and $\Omega_0'$ its symmetric strip.  We then divide $\Omega_0$ ( resp. $\Omega'_0$) into three smaller strips  such that each of them has the width $\frac{\xi}{3(8\Sigma+4)}$.
Let $\Omega_{00}$ (resp. $\Omega'_{00}$) be the middle $\frac{1}{3}$ part of  $\Omega_0$ (resp. $\Omega'_0$). Thus for any $x+iy\in \Omega_{00}$ (and of course $x-iy\in\Omega_{00}'$)
$$|v_j(x\pm iy)-E|=|g_{j,E}(x\pm iy)|\prod_{m=1}^{k_j}|v_j(x\pm iy)-z_{j,m}|^{n_{j,m}}\geq q\left(\frac{\xi}{24\Sigma+12}\right)^{\Sigma},$$
which completes the proof.
\end{proof}

We can now state our main result of this section.

\begin{prop}\label{average}
Let $\omega\in\mathrm{DC}_t$. There are $\epsilon_0=\epsilon_0(V,\rho,l)>0$ (depending only on $V,\rho,l$) and $C=C(V,\rho)>0$ (depending only on $V,\rho$) such that if $0<\epsilon\leq \epsilon_0$, then  for all  $E$ in a compact interval $\mathcal{A}$ and  $N$ large enough (depending on $V,l,\epsilon,\rho,t$),  we have
\begin{equation*}
\frac{1}{N}\int_{\mathbb{T}}\log|\det[H_N(x,E)]|\mathrm{d}x \geq\sum_{j=1}^l \int_{\mathbb{T}}\log|v_j(x)-E|\mathrm{d}x -Cl\epsilon^{\frac{1}{2\Sigma}}.
\end{equation*}
where $\Sigma$ is given by  Lemma \ref{v-e}.
\end{prop}

\begin{rem}
From Lemma \ref{logv}, we have
\begin{equation*}\label{min}
\min_{E}\sum\limits_{j=1}^{l}\int_0^1\log|v_j(x)-E|\mathrm{d}x>-C>-\infty,
\end{equation*}
where $C>0$ depends on $v_1,\cdots,v_l$ only.
\end{rem}

\begin{proof}
%Since $v_j$ is non-constant, from Lemmata \ref{mes} and \ref{logv}, we have
%\begin{equation}\label{min}
%\min_{E}\int_0^1\log|v_j(x)-E|>-C(v_j).
%\end{equation}

We first consider the upper bound on determinant of the matrix $H_N(z,E)=H_N(z)-EI_N$ on the strip $\Delta_\rho$. By  Hadamard's inequality, one has
\begin{equation*}
  |\det[H_N(z,E)]|\leq\prod_{1\leq j\leq l}\prod_{1\leq n\leq N}(|v_j(z+n\omega)-E|+C\epsilon),
\end{equation*}
where $C>0$ is a constant depending only on $\rho$. Thus when $E$ is in a compact interval $\mathcal{A}$, we have
\begin{align} \label{upper}
 \frac{1}{N}\log|\det[H_N(z,E)]|
 %\leq \frac{1}{N}\sum_{1\leq j\leq l}\sum_{1\leq n\leq N}
 %              \log(|v_j(z+n\omega)-E|+C(\rho)\epsilon) \\
 \leq C(V,\rho)l.
\end{align}

Next, we consider a lower bound on $|\det[H_N(x\pm iy,E)]|$ for some $y\in(0,\rho)$.  We fix $\xi=\epsilon^{\frac{1}{2\Sigma}}$, where $\Sigma$ is given by Lemma \ref{v-e}.  Then $\xi\in(0,\frac{\rho}{2})$ for  $0<\epsilon<\left(\frac{\rho}{2}\right)^{2\Sigma}$. From Lemma \ref{v-e}, there is $y_0\in(\frac{\xi}{2},\xi)$ such that
\begin{equation}\label{y0}
 \min_{E\in\mathbb{R}}\min_{1\leq j\leq l}\min_{x\in\mathbb{R}}|v_j(x\pm iy_0)-E|\geq c\xi^{\Sigma}.
\end{equation}
We write
\begin{align}\label{hn-e}
\nonumber  H_N(x\pm iy_0,E) & =D_N(x\pm iy_0)+\epsilon B_N\\
 % &=(H_N(x+iy_0)-E)D_n^{-1}(x+iy_0)D_N(x+iy_0) \\
    & =\left(I_N+\epsilon B_N D_N^{-1}(x\pm iy_0)\right)D_N(x\pm iy_0),
\end{align}
% H_N(x+iy_0)-E=(H_N(x+iy_0)-E)D_n^{-1}D_N(x+iy_0)=\left(I_{Nl}+\epsilon B_N D_N^{-1}(z)\right)D_N(z),
where
\begin{equation*}
  D_N(x\pm iy_0)=\text{diag}[V(x\pm iy_0+\omega)-EI_N,\cdots,V(x\pm iy_0+N\omega)-EI_N]
\end{equation*}
and
\begin{equation*}
  B_N=\left(
           \begin{array}{cccc}
             W_0& W_{-1} & \cdots & W_{-N+1} \\
             W_1 & W_0 & \ddots & \vdots \\
             \vdots & \ddots & \ddots & W_{-1} \\
             W_{N-1} & \cdots & W_1 & W_0 \\
           \end{array}
         \right).
\end{equation*}
Then it follows that
\begin{align*}
  \frac{1}{N}\log|\det[H_N(x\pm iy_0,E)]|=&\ \frac{1}{N}\log|\det D_N(x\pm iy_0)|  \\
   &\ \  + \frac{1}{N}\log|\det[I_N+\epsilon B_N D_N^{-1}(x\pm iy_0)]|.
\end{align*}
Hence when $N$ is large enough, from Lemma \ref{dk} and (\ref{y0}), it follows that
\begin{align}\label{dn}
\nonumber \frac{1}{N}\log|\det D_N(x\pm iy_0)| =&\ \frac{1}{N} \sum_{j=1}^l \sum_{n=1}^{N}\log[|v_j(x\pm iy_0+n\omega)-E|+\epsilon]\\
  \nonumber   &\ \  - \frac{1}{N}\sum_{j=1}^l \sum_{n=1}^{N}\log \left(1+\frac{\epsilon}{|v_j(x\pm iy_0+n\omega)-E|}\right)\\
% =&\ \sum_{j=1}^l \int_0^1\log(|v_j(x+iy_0)-E|+\epsilon)\text{d}x+O(N^{-\delta_1})\\
 %     &\ \   -\frac{1}{N}\sum_{j=1}^l\sum_{n=1}^{N}\log\left(1+\frac{\epsilon}{|v_j(x+iy_0+n\omega)-E|}\right)\\
 %\geq&\ \sum_{j=1}^l\int_0^1\log(|v_j(x+iy_0)-E|+\epsilon)\text{d}x+l\mathcal{O}(N^{-\delta_1})\\
 %&\ \  -\frac{l\epsilon}{ct^{\Sigma}}\\
 \geq&\ \sum_{j=1}^l\int_0^1\log(|v_j(x\pm iy_0)-E|+\epsilon)\text{d}x-Cl\epsilon^{\frac{1}{2}}.
\end{align}
% \frac{1}{N}\log|\det[H_N(z)-E]|=\frac{1}{N}\log|\det D_N(z)|
 %                   + \frac{1}{N}\log|\det[I_{Nl}+\epsilon B_N D_N^{-1}(z)]|.
On the other hand, from (\ref{y0}),  one has
$$\|\epsilon B_ND_N^{-1}(x\pm iy_0)\|\leq  C\epsilon^{\frac{1}{2}},$$
and for $0<\epsilon\ll1$ (depending only on $v_1,\cdots,v_l,\rho$), $\|\epsilon B_ND_N^{-1}(x\pm iy_0)\|<\frac{1}{2}.$ Therefore when $0<\epsilon\ll 1$, by Hadamard's inequality and Neumann expansion technique, it follows that
\begin{align}\label{bd-1}
 \nonumber \log|\det[I_N+\epsilon B_ND_N^{-1}(x\pm iy_0)]|
  %= & -\log|\det[I+\epsilon B_ND_N^{-1}(x+iy_0)]^{-1}|\\
=&-\log|\det[I_N+\sum_{s\geq 1}(-1)^s(\epsilon B_N D_N^{-1}(x\pm iy_0))^s ]|\\
%\geq &-\log \prod_{n=1}^{Nl}[1+\sum_{s\geq 1}\epsilon^s \|(B_ND_N^{-1}(x+iy_0))\|_2]\\
\nonumber\geq &-\sum_{n=1}^{Nl}\log (1+2\epsilon\|B_ND_N^{-1}(x\pm iy_0){\delta}_n\|)\\
%nonumber\geq &-\sum_{j=1}^l\sum_{n=1}^{N}\log (1+C(\rho)\epsilon|v_j(x+iy_0+n\omega)-E|^{-1})\\
>& -CNl\epsilon^{\frac{1}{2}}.
\end{align}
Combing (\ref{hn-e})-(\ref{bd-1}), for $0<\epsilon\leq \epsilon_0$ (depending only on $V,\rho$),  $N$  large enough (depending on $V,l,\epsilon,\rho,t$) and  any $x\in \mathbb{T}$, we have
\begin{equation}\label{lower}
\frac{1}{N}\log|\det[H_N(x\pm iy_0,E)]|\geq \sum_{j=1}^l\int_0^1\log(|v_j(x\pm iy_0)-E|+\epsilon)\text{d}x-Cl\epsilon^{\frac{1}{2}}.
\end{equation}

Finally, in order to get a lower bound on $\int_{\mathbb{T}}\log|\det[H_N(x,E)]|\mathrm{d}x $, we exploit subharmonicity of the function
\begin{equation}\label{u}
  u(z)=u_N(z):=\frac{1}{N}\log|\det[H_N(z,E)]|, \ z\in \Delta_{\rho}.
\end{equation}
Fix $x\in\mathbb{T}$  and denote $y_1:=\frac{\rho}{2}$, $\Omega_{\rho}:=\{z\in\mathbb{C}:0\leq\Im z\leq y_1\}$. We use harmonic measure estimates of  Bourgain and  Goldstein \cite{BGA} here and we include the basic properties of the harmonic measure in \S2.2 for reader's convenience. Since $u(z)$ is subharmonic, we have
\begin{align*}
   u(x+iy_0)\leq  & \int_{\{z:\Im z=0\}}u(z)\text{d}\mu(x+iy_0,z,\Omega_{\rho})+\int_{\{z:\Im z=y_1\}}u(z)\text{d}\mu(x+iy_0,z,\Omega_{\rho}) \\
   = &\int_{-\infty}^{+\infty} u(t)\text{d}\mu(x+iy_0,t,\Omega_{\rho})
       + \int_{-\infty}^{+\infty} u(t+iy_1)\text{d}\mu(x+iy_0,t,\Omega_{\rho}) \\
   =  &\int_{-\infty}^{+\infty} u(x+t)\text{d}\mu(iy_0,t,\Omega_{\rho})
       + \int_{-\infty}^{+\infty} u(x+t+iy_1)\text{d}\mu(iy_0,t,\Omega_{\rho}) ,\\
\end{align*}
where $\mu$ is the harmonic measure defined in \S 2.2 and the last equality follows from (\ref{hmi}). Consequently, using Fubini's theorem, one has
\begin{align*}
  \int_0^1 u(x+iy_0)\text{d}x \leq &\ \mu(iy_0,\{z:\Im z=0\}, \Omega_{\rho}) \int_0^1 u(x)\text{d}x \\
          & +\mu(iy_0,\{z:\Im z=y_1\}, \Omega_{\rho}) \int_0^1 u(x+iy_1)\text{d}x.\\
\end{align*}
Choose a conformal mapping $\phi:\Omega_{\rho}\rightarrow\mathbb{H},z\mapsto e^{\frac{2\pi}{\rho}z}$. From  (\ref{mu1}) and (\ref{mu3}), it follows that
\begin{align*}
  \int_0^1 u(x+iy_0)\text{d}x\leq&\ \mu(e^{i\frac{2\pi}{\rho}y_0},[0,+\infty), \mathbb{H}) \int_0^1 u(x)\text{d}x \\
                      &\ \  +\mu(e^{i\frac{2\pi}{\rho}y_0},(-\infty,0], \mathbb{H}) \int_0^1 u(x+iy_1)\text{d}x.\\
  =&\left(1-\frac{y_0}{y_1}\right) \int_0^1 u(x)\text{d}x +\frac{y_0}{y_1} \int_0^1 u(x+iy_1)\text{d}x.\\
\end{align*}
Recalling (\ref{upper}) and (\ref{lower}), we have
\begin{align*}
  \sum_{j=1}^l\int_0^1\log(|v_j(x+iy_0)-E|+\epsilon)\text{d}x-Cl\epsilon^{\frac{1}{2}}
   \leq \left(1-\frac{y_0}{y_1}\right) \int_0^1 u(x)\text{d}x +C\frac{y_0}{y_1}.
\end{align*}
Hence,
\begin{align*}
\int_0^1 u(x)\text{d}x \geq &\ \frac{y_1}{y_1-y_0} \left(\sum_{j=1}^l \int_0^1\log|v_j(x+iy_0)-E|\text{d}x
               -Cl\epsilon^{\frac{1}{2}}- C\frac{y_0}{y_1}\right) \\
          \geq& \sum_{j=1}^l\int_0^1\log|v_j(x+iy_0)-E|\text{d}x- C \frac{2\xi}{\rho-\xi}-Cl\frac{\rho}{\rho-\xi}\epsilon^{\frac{1}{2}}\\
          \geq& \sum_{j=1}^l\int_0^1\log|v_j(x+iy_0)-E|\text{d}x-C l(\xi+\epsilon^{\frac{1}{2}}).
\end{align*}
Repeating the process above, we have
\begin{align*}
\int_0^1 u(x)\text{d}x \geq& \sum_{j=1}^l\int_0^1\log|v_j(x-iy_0)-E|\text{d}x-C l(\xi+\epsilon^{\frac{1}{2}}).
\end{align*}
From the convexity argument (see \cite{JMC} for details), we obtain
\begin{equation*}
  \int_0^1 \log|v_j(x)-E|\text{d}x\leq \frac{1}{2}\int_0^1 \log|v_j(x+iy_0)-E|\text{d}x +\frac{1}{2}\int_0^1 \log|v_j(x-iy_0)-E|\text{d}x.
\end{equation*}
Noting $\xi=\epsilon^{\frac{1}{2\Sigma}}$, we have
\begin{equation*}
\frac{1}{N}\int_0^1 \log|\det[H_N(x,E)]|\text{d}x \geq \sum_{j=1}^l \int_0^1 \log|v_j(x)-E|\text{d}x -Cl\epsilon^{\frac{1}{2\Sigma}}.
\end{equation*}
The proof of Proposition \ref{average} is finished.
\end{proof}

\section{Green's function estimates}

In this section, we consider the  Green's function
$$G_{\mathcal{N}}(x;E):=(H_{\mathcal{N}}(x)-EI_{\mathcal{N}})^{-1},$$
whenever $H_{\mathcal{N}}(x)-EI_{\mathcal{N}}$ is invertible, where $\mathcal{N}\subset \mathbb{Z}$ being an interval. A crucial ingredient in the proof of AL is the so called LDT for Green's function. In this section, we will prove the LDT and this can be achieved using a quantitative Birkhoff ergodic theorem for the function $u_N(x)$ defined in (\ref{u}) as well as the uniformly upper and averaging lower bounds in Propositions \ref{upperbound} and \ref{average}.

\begin{defn} We say that $G_{\mathcal{N}}(x;E)$ is a \emph{good} Green's function if for all
 $n,n'\in\mathcal{N}$,
\begin{equation}\label{good}
  \|G_{\mathcal{N},(n,n')}(x;E)\|\leq e^{-(|n-n'|-\frac{|\mathcal{N}|}{100})(\rho-\epsilon^{\delta})},
\end{equation}
where $\delta>0$ is a constant.
\end{defn}
\begin{rem}
It is easy to show $G_{\mathcal{N}}(x;E)$ is a \emph{good} Green's function if and only if there is some $C=C(l)>0$ (depending only on $l$) such that
$$|G_{\mathcal{N},(\alpha,\alpha')}(x;E)|\leq Ce^{-\left(|n(\alpha)-n(\alpha')|-\frac{|\mathcal{N}|}{100}\right)(\rho-\epsilon^{\delta})},$$
where $(a-1)l <\alpha,\alpha'\leq bl$ with $\mathcal{N}=[a,b]\subset \mathbb{Z}$  being an interval.
\end{rem}

We first recall a useful lemma concerning the semi-algebraic set (for the basic knowledge of the semi-algebraic sets, see Chapter 9 of \cite{BB} by  Bourgain).
\begin{lem}[Lemma 9.7 in \cite{BB}]\label{semi}
 Let $\mathcal{S}\subset[0,1]$ be a semi-algebraic set of degree $B$ and  $\mathrm{Leb}(\mathcal{S})\leq\eta$. Let $\omega\in \mathrm{DC}_t$ and $K$ be a large integer satisfying
$$\log B\ll\log K<\log \frac{1}{\eta}. $$
Then for any $x\in\mathbb{T}$,
\begin{equation}
\#\{k=1,\cdots,K: x+k\omega\in\mathcal{S}\ (\mathrm{mod}\ 1) \}\leq K^{1-\tau},
\end{equation}
where $\tau=\tau(t)\in(0,1)$ depends only on $t$ and $\# A$ denotes the number of elements of finite set $A$.
\end{lem}

One also has the following important quantitative Birkhoff ergodic theorem for some subharmonic function.
\begin{lem}[Theorem 6.5 in \cite{DKB}] \label{ldt}
Let $u:\Delta_\rho\to [-\infty,\infty)$ be a subharmonic function satisfying \begin{equation}\label{ucd}
\sup_{z\in\Delta_\rho}u(z)+\left(\int_{\mathbb{T}}|u(x)|^2\mathrm{d}x\right)^{\frac{1}{2}}\leq C_\star.\end{equation}
Let $\omega\in\mathrm{DC}_{t}, M_0=t^{-2}$. Then there are absolute constant $a>0$ and $C=C(\rho)>0$ (depending only on $\rho$) such that for $M\geq M_0$,
\begin{equation}\label{ldte}
    \mathrm{Leb}\left\{x\in\mathbb{T}:\left|\frac{1}{M}\sum_{j=0}^{M-1}u(x+j\omega)-\int_{\mathbb{T}}u(x)\mathrm{d}x\right|\geq \frac{CC_\star}{\rho}M^{-a}\right\}\leq e^{-\frac{20a\rho}{CC_\star}M^a}.
\end{equation}
\end{lem}

In the following, we let $u=u_N(x)$ with $u_N(x)$ being given by (\ref{u}). Recalling Proposition \ref{average}, we can define the set

\begin{align}\label{bnm}
 \nonumber &\ \mathcal{B}_N^M= \mathcal{B}_N^M(\omega,E)\\
  =\ &\left\{x\in\mathbb{T}:\frac{1}{M}\sum_{j=0}^{M-1}u(x+j\omega)\leq \sum_{j=1}^l \int_{\mathbb{T}}\log|v_j(x)-E|\mathrm{d}x-Cl\epsilon^{\delta} \right\},
\end{align}
where $\omega\in\mathrm{DC}_t$, $E$ belongs to an interval $\mathcal{A} \subset \mathbb{R}$ and $C>0$ depends only on $V,\rho,l$. We assume $\mathrm{Leb}(\mathcal{A})\leq C(V,\rho)$. We have the following proposition.

\begin{prop}\label{ldt}
 Let $\mathcal{B}_N^M$ be defined by (\ref{bnm}) and $0<\epsilon\ll1$ (depending only on $V,\rho$). There are absolute constants $a > 0$ and $P\in\mathbb{N}$ such that if $M\geq M_0(V,\rho,\epsilon,l,t)$ and $N\geq N_0(V,\rho,\epsilon,l,t)$, then the following holds.
\begin{itemize}
\item[(i)]
\begin{equation*}
  \emph{Leb}(\mathcal{B}_N^M(\omega,E))<e^{-cM^a},
\end{equation*}
where $c=\frac{20\rho a}{CC_\star}>0$ is given by Lemma \ref{ldt}.
\item[(ii)] For every $x\notin \mathcal{B}_N^M(\omega,E)$ there is $0\leq j< M$ such that $G_N(x+j\omega;E)$ is a good Green's function.

\item[(iii)] For any $x\in\mathbb{T}$,
\begin{equation*}
  \#\left\{0\leq n < N^P: G_N(x+n\omega;E)\text{ is not a good Green's function} \right\}\leq N^{(1-\tau)P},
\end{equation*}
where $\tau=\tau(t)$ is given by Lemma \ref{semi}.

\end{itemize}
\end{prop}

\begin{proof}
(i) We use Lemma \ref{ldt} for $u(x)=u_N(x)$ defined in (\ref{u}). From Proposition \ref{average}, (\ref{upper}) and Lemma \ref{logv}, we can take $C_\star=C(V,\rho)l$ and if $M\geq M_0$, we have $\mathrm{Leb}(\mathcal{B}_1)\leq e^{-cM^a} $ with
$$\mathcal{B}_1=\left\{x\in\mathbb{T}:\left|\frac{1}{M}\sum_{j=0}^{M-1}u(x+j\omega)-\int_{\mathbb{T}}u(x)\text{d}x\right| \geq\frac{C_\star C}{\rho} M^{-a}   \right\}.$$
Thus it suffices to show $\mathcal{B}_N^M\subset \mathcal{B}_1 $ if $M\gg1$. In fact, if $x\notin\mathcal{B}_1$, then by Proposition \ref{average} and for $M\gg1$, we have
\begin{align*}
  \frac{1}{M}\sum_{j=0}^{M-1}u(x+j\omega) \geq &\  \int_{\mathbb{T}}u(x)\text{d}x -\frac{C_\star C}{\rho}M^{-a}\\
%\geq &\ \sum_{j=1}^l \int_{\mathbb{T}}\log|v_j(x)-E|\text{d}x -Cl\epsilon^{\delta}-Cl\epsilon\\
\geq &\ \sum_{j=1}^l \int_{\mathbb{T}}\log|v_j(x)-E|\text{d}x -C_1l\epsilon^{\delta},
\end{align*}
that is $x\notin \mathcal{B}^M_N$.

(ii) By Cramer's rule,  we have
\begin{equation*}
  G_{N,(\alpha,\alpha')}(x;E)=\frac{\mu_{N,(\alpha,\alpha')}(x,E)}{\det[H_N(x)-EI]},
\end{equation*}
and thus
\begin{equation}\label{1n}
  \frac{1}{N}\log|G_{N,(\alpha,\alpha')}(x;E)|
  %& = \frac{1}{N}\log|\mu_{N,(\alpha,\alpha')}(x,E)| -\frac{1}{N}\log|\det[H_N(x)-EI]|\\
    = \frac{1}{N}\log|\mu_{N,(\alpha,\alpha')}(x,E)|-u_N(x).
\end{equation}
From Proposition \ref{upperbound}, we have for any $x\in\mathbb{T}$ and $N\geq N_0$,
\begin{equation}\label{1n1}
  \frac{1}{N}\log|\mu_{N,(\alpha,\alpha')}(x,E)|\leq \sum_{j=1}^l \int_{\mathbb{T}}\log|v_j(x)-E|\text{d}x + l\epsilon^{\sigma} +\frac{|n(\alpha)-n(\alpha')|}{N}(-\rho+\epsilon^{\sigma}).
\end{equation}
If $x\notin \mathcal{B}_N^M$, by item (i), we have
\begin{equation*}
 \frac{1}{M}\sum_{j=0}^{M-1}u(x+j\omega)\geq\sum_{j=1}^l \int_{\mathbb{T}}\log|v_j(x)-E|\text{d}x -Cl\epsilon^{\delta}.
\end{equation*}
Then there is  $0\leq  j_0< M$, such that
\begin{equation}\label{j0}
  u_N(x+j_0\omega)\geq \sum_{j=1}^l \int_{\mathbb{T}}\log|v_j(x)-E|\text{d}x -2Cl\epsilon^{\delta}.
\end{equation}
Thus  combing (\ref{1n})-(\ref{j0}), we conclude for $0<\epsilon\ll1$ (depending only on $V,\rho,l$),
\begin{align*}
 \frac{1}{N}\log|G_{N,(\alpha,\alpha')}(x+j_0\omega;E)|
 %\leq & \ \sum_{j=1}^l \int_{\mathbb{T}}\log|v_j(x)-E|\text{d}x -(\rho-\epsilon^{\delta_2})\frac{|n(\alpha)-n(\alpha')|}{N}\\
 % & \ \ +Cl \epsilon^{\delta_1} - \sum_{j=1}^l \int_{\mathbb{T}}\log|v_j(x)-E|\text{d}x +Cl\epsilon^{\delta}\\
 \leq & -(\rho-\epsilon^{\sigma})\frac{|n(\alpha)-n(\alpha')|}{N} + Cl\epsilon^{\min\{\sigma,\delta\}}.
\end{align*}
This implies $G_N(x+j_0\omega;E)$ is a $good$ Green's function (see (\ref{good})).

(iii) Let $1\ll M\sim N^{\frac{1}{2}}$. Fixing $\omega\in \mathrm{DC}_t$ and $E\in\mathcal{A}$, we rewrite the inequality  defining the set $\mathcal{B}_N^M$ as
\begin{equation}\label{bnm2}
  Q_1^\pm(x):=\pm\prod_{j=0}^{M-1}\text{det}[H_N(x+j\omega)-EI_N]\leq e^{NM(\sum\limits_{j=1}^l \int_{\mathbb{T}}\log|v_j(x)-E|\text{d}x-Cl\epsilon^\delta)}.
\end{equation}
We truncate each $v_j(x)$ as $v_{j,N}(x)=\sum\limits_{|k|\leq N^2}\widehat{v}_j(k)e^{2\pi kix}$ first, where $\widehat{v}_j(k)$ are the corresponding Fourier coefficients. As a result, $\|v_j-v_{j,N}\|_0\leq Ce^{-\rho N^2}$ ($1\leq j\leq l$). If we replace each $v_j$ with $v_{j,N}$ in $Q_1^\pm$, we get the  trigonometric polynomials $Q_2^\pm(\cos2\pi x,\sin2\pi x)$  and the degrees of $Q_2^\pm$ are bounded by $lMN^3$. Obviously,  $\|Q_2^{\pm}-Q_1^{\pm}\|_0\leq Ce^{-\rho{N^{2-}}}$.  Furthermore, in $ Q_2^\pm$,  if we replace $\sin2\pi x,\cos2\pi x$ with polynomials (in $x$) of degree $N$ , we can obtain polynomials $Q_3^\pm(x)$ (in $x$) of degrees being bounded by $\mathcal{O}(N^4M)$ such that $\|Q_3^{\pm}-Q_2^{\pm}\|_0\leq Ce^{-\rho{N^{2-}}}$. Thus we have showed  $\mathcal{B}_{N}^M(\omega,E)$ can be replaced by a semi-algebraic set
$$\mathcal{S}=\left\{x\in[0,1]: Q_3^{\pm}(x)\leq e^{NM(\sum\limits_{j=1}^l\int_{\mathbb{T}}\log|v_j(x)-E|\text{d}x-Cl\epsilon^\delta)}\right\}, $$
and the statements of items (i) (ii) are still valid for $\mathcal{S}$ (since $MN\ll N^{2-}$ and Lemma \ref{logv}). In fact, we know the degree of $\mathcal{S}$ is bounded by $\mathcal{O}(N^5)$. Thus using Lemma \ref{semi}, there is some absolute constant $P\in\mathbb{N}$, such that
\begin{equation*}
  \#\{k=0,1,\cdots,N^P-1:x+k\omega\in\mathcal{S}\}\leq N^{(1-\tau)P}.
\end{equation*}
Since $\mathcal{S}$ has property (ii), then  (iii) holds.
\end{proof}

%\begin{remark}\label{remark}
 %In the proof of localization we will apply item (ii) in Proposition \ref{ldt} with $M$ being a small fraction of $N$, say $M=\frac{1}{100}N$. Then the statement may be reformulated as follows.

%There is an absolutely constant $a>0$ and for every large enough integer $N$ there is a set of phases $\mathcal{B}_N=\mathcal{B}_N(\omega,E)$ with $\text{meas}(\mathcal{B}_N(\omega,E))<e^{-N^a}$ such that if $x\notin \mathcal{B}_N(\omega,E)$, then $G_N(x+j\omega;E)$ is a good Green's function for some $0\leq j<\frac{1}{100}N$.

%With this reformulation, we have the analogue of Proposition 7.19 in J. Bourgain's monograph \cite{Bourgain2005Green}. That proposition forms the basis for the proof of localization for quasi-periodic Schr\"odinger operators described in Chapter 10 of the monograph.
%\end{remark}
\begin{rem}
One should note that the set $\mathcal{B}_{N}^M(\omega,E)$ relies heavily on $E$.
\end{rem}
\begin{rem}\label{-n}
The statements similar to those in items (ii) (iii) also hold if we replace $G_{N}(x;E)$ with $G_{[-N,N]}(x;E)$.
\end{rem}

\section{The proof of localization: eliminating the energy}

Based on the Green's function estimates  and semi-algebraic sets considerations in Proposition \ref{ldt}, we will finish the proof of Theorem \ref{mthm}. We use the techniques developed by Bourgain and Goldstein in \cite{BGA} to establish non-perturbative AL for quasi-periodic Schr\"odinger operators.

From Sch'nol-Simon theorem (see \cite{HA} for details),  to prove AL, it suffices to show that every extended state decays exponentially.
%From Remark \ref{rem1}, we may consider the decaying properties of extended states.
\begin{defn}
We call $E\in\mathbb{R}$ a generalized eigenvalue of $H_{\epsilon,\omega}(x)$ if there is some $\vec{\psi}=\{\vec{\psi}_n\}_{n\in \mathbb{Z}}$ with $\|\vec{\psi}_n\|\leq C(\vec{\psi})(1+|n|^2)$ such that $H_{\epsilon,\omega}(x)\vec{\psi}=E\vec{\psi}$. Moreover, the corresponding $\vec{\psi}$ is called the
extended state.
\end{defn}

We note that
\begin{itemize}
\item Let $E$ be a generalized eigenvalue of $H_{\epsilon,\omega}(x)$ and $\vec{\psi}$ be the corresponding extended state. Then for any $j\in\mathcal{N}\subset\mathbb{Z}$,
\begin{equation}\label{identity}
  \vec{\psi}_j= -\sum_{i\in \mathcal{N},k\notin\mathcal{N}}G_{\mathcal{N},(j,i)}(x;E) W_{i-k}\vec{\psi}_k.
\end{equation}
\item For any $\mathcal{N}\subset\mathbb{Z}$,
\begin{equation}\label{ivariant}
  G_{\mathcal{N}}(x+j\omega;E)=G_{\mathcal{N}+j}(x;E),
\end{equation}
where $\mathcal{N}+j=\{n+j:n\in\mathcal{N}\}$.
\end{itemize}

%Let , $E$
%\begin{equation*}
 % (H_{\mathcal{N}}(x)-EI)\vec{\psi}_{\mathcal{N}}=-\left(
                           % \begin{array}{c}
                             % \sum_{k\notin\mathcal{N}} W_{-k+1}\vec{\psi}_k\\
                             % \sum_{k\notin\mathcal{N}} W_{-k+2}\vec{\psi}_k\\
                             % \vdots \\
                             % \sum_{k\notin\mathcal{N}} W_{-k+N}\vec{\psi}_k \\
                           % \end{array}
                         % \right).
%\end{equation*}
%Hence,
%\begin{equation*}
  %\vec{\psi}_{\mathcal{N}}=-G_{\mathcal{N}}(x;E)\left(
                           % \begin{array}{c}
                              %\sum_{k\notin\mathcal{N}} W_{1-k}\vec{\psi}_k\\
                              %\sum_{k\notin\mathcal{N}} W_{2-k}\vec{\psi}_k\\
                              %\vdots \\
                              %\sum_{k\notin\mathcal{N}} W_{N-k}\vec{\psi}_k \\
                            %\end{array}
                          %\right).
%\end{equation*}
%It follows that for $j\in\mathcal{N}$,
In fact, $good$ Green's function implies exponential decay of the corresponding extended state.

\begin{lem}\label{psiN}
Let $N\geq N_0$, $\mathcal{N}=[\sqrt{N},2N]$. Assume  $E$ is a generalized eigenvalue of $H_{\epsilon,\omega}(x)$ with $\vec{\psi}=\{\vec{\psi}_n\}$ being corresponding extended state. If $G_{\mathcal{N}}(x;E)$ is a good Green's function, then for $\frac{N}{2}\leq j\leq N$,
\begin{equation*}
  \|\vec{\psi}_j\|\leq e^{-\frac{\rho}{3}j}.
\end{equation*}
\end{lem}

\begin{proof}
It is easy to show $$\sum\limits_{k\geq a>0}k^pe^{-\rho k}\leq Ca^pe^{-\rho a}, \ p\geq 0,$$
where $C>0$ depends only on $\rho$ and $p$. Then for any $i\in\mathcal{N}$,
\begin{align*}
  \sum_{k\notin \mathcal{N}}e^{-\rho|k-i|}|k|^2 \leq &\ 2\sum_{k\notin \mathcal{N}}e^{-\rho|i-k|}(|i-k|^2+i^2) \\
    \leq &\ CN^2(e^{-\rho|i-\sqrt{N}|}+e^{-\rho|i-2N|}).
\end{align*}
As a result, recalling (\ref{identity}), we have for any $\frac{N}{2}\leq j\leq N$,
\begin{align*}
  \|\vec{\psi}_j\|\leq &\ CN^2 \sum_{i\in\mathcal{N}}e^{-(\rho-\epsilon^\delta)(|j-2N|-\frac{|\mathcal{N}|}{100})}+CN^2 \sum_{i\in\mathcal{N}}e^{-(\rho-\epsilon^\delta)(|j-\sqrt{N}|-\frac{|\mathcal{N}|}{100})}\\
  \leq &\ e^{-\frac{\rho}{3}j} \ \ \mbox {(for $N\gg1$)}.
\end{align*}
The proof is finished.
\end{proof}

The following lemma suggests that a $good$  Green's function at large scale can be obtained from paving $good$ Green's functions at small scale.
\begin{lem}[Lemma 10.33 in \cite{BB}]\label{paving}
Let $\mathcal{N}\subset\mathbb{Z}$ be an interval with length $N$ and $\{\mathcal{N}_{\alpha}\}$ be subintervals with length $M\ll N$. Assume
\begin{itemize}
\item [(i)] If ${k\in \mathcal{N}}$, then there is some $\alpha$ such that $[k-\frac{M}{4},k+\frac{M}{4}]\cap \mathcal{N}_\alpha\subset \mathcal{N}$,

\item [(ii)] For all $\alpha$, $G_{\mathcal{N}_{\alpha}}$ are good.
\end{itemize}
%\begin{equation*}
 % |G_{\mathcal{I}_{\alpha}}(n_1,n_2)|<e^{-c_0|n_1-n_2|}\  \mathrm{for} \ n_1,n_2\in \mathcal{I}_{\alpha},|n_1-n_2|>\frac{M}{10}.
%\end{equation*}
Then $G_{\mathcal{N}}$ is good.
%\begin{equation*}
 %\|G_{\mathcal{I}}\|<e^{M} \text{ and } |G_{\mathcal{I}}(n_1,n_2)|<e^{-c_1|n_1-n_2|}\  \mathrm{for}\  n_1,n_2\in \mathcal{I},|n_1-n_2|>\frac{N}{10},
%\end{equation*}
%where $0<c_0-c_1\ll1$
\end{lem}
\begin{rem}
This lemma follows from the resolvent identity and see \S15 of \cite{BGA} for details.
\end{rem}

 We also need the following  lemma which is crucial in the eliminating energy process.
\begin{lem}\label{g-}
Let $N\geq{N}_0$ be fixed. Then there are some absolute constant $P_1>0$ and $N_1\in\mathbb{N}$ with $N_1=\mathcal{O}(N^{P_1})$ such that if  $E$ is a generalized eigenvalue of $H_{\epsilon,\omega}(x)$ and $\vec{\psi}$ is the corresponding extended state with $\|\vec{\psi}_0\|=1$,
%$\vec{\psi}=\{\vec{\psi}_n\}_{n\in\mathbb{Z}}$ satisfies $H_{\epsilon,\omega}(x)\vec{\psi}=E\vec{\psi}$, $\|\vec{\psi}_n\|\leq C(1+|n|^2)$ and $\|\vec{\psi}_0\|=1$,
then
\begin{equation*}
  \emph{dist}(E,\sigma(H_{[-N_1,N_1]}(x)))\leq e^{-\frac{\rho}{2} N^{2}},
\end{equation*}
where $\sigma(H)$ denotes the spectrum of the operator $H$.
\end{lem}

\begin{proof}
The proof is similar to that of (6.6) in \cite{BA}.  Using (iii) of Proposition \ref{ldt} and Remark \ref{-n} at scale $N^2$, we obtain that there is some interval $\mathcal{I}\subset [0,N^{2P}]$ with length $|\mathcal{I}|\sim N^{2\tau P}$ such that for $n\in \mathcal{I}\cup (-\mathcal{I})$, the Green's function $G_{[-N^2,N^2]+n}(x;E)=G_{[-N^2,N^2]}(x+n\omega;E)$ is $good$, where $-\mathcal{I}=\{-n:n\in\mathcal{I}\}$. As a result, one has for any $n\in \mathcal{I}\cup (-\mathcal{I})$,
\begin{align*}
  ||\vec{\psi}_n||\leq&\ C\sum\limits_{|j_1-n|\leq N^2, |j_2-n|>N^2}e^{-(\rho-\epsilon^\delta)(|j_1-n|-\frac{N^2}{50})}e^{-\rho|j_1-j_2|}|j_2|^2\\
 = & \ C\sum_{|j_1-n|\leq{N^2}, |j_2-n|>N^2}e^{\frac{\rho N^2}{25}}e^{-\rho|j_2-n|}|j_2|^2\ \ \mbox{(for $0<\epsilon\ll1$)}\\
 \leq &\ e^{-\frac{2\rho}{3}N^2}.
\end{align*}
We now let $N_1$ be the center of interval $\mathcal{I}$. Then
\begin{align*}
  1=\|\vec{\psi}_0\|\leq & \ \sum_{|j_1|\leq N_1, |j_2|>N_1}\|G_{[-N_1,N_1],(0,j_1)}(x;E)\|e^{-\rho|j_1-j_2|}\|\vec{\psi}_{j_2}\|\\
 \leq & \ \|G_{[-N_1,N_1]}(x;E)\| \sum_{|j_1|\leq N_1, j_2\in \mathcal{I}\cup (-\mathcal{I})}e^{-\frac{2\rho}{3}N^2}\\
 &+C\|G_{[-N_1,N_1]}(x;E)\| \sum_{|j_1|\leq N_1, |j_2|>N_1+\frac{|\mathcal{I}|}{2}}e^{-\rho|j_1-j_2|}|j_2|^2\\
\leq&\ \|G_{[-N_1,N_1]}(x;E)\| e^{-\frac{\rho}{2}N^2}.
\end{align*}
Thus
\begin{equation*}
  \text{dist}(E,\sigma(H_{[-N_1,N_1]}(x)))=\|G_{[-N_1,N_1]}(x;E)\|^{-1} \leq e^{-\frac{\rho}{2}N^2}.
\end{equation*}
\end{proof}

Now we can finish the proof of our main theorem.

\begin{proof}[{\bf Proof of Theorem \ref{mthm}}]

Fix $x_0\in\mathbb{T}$ and let $N\in\mathbb{N}$ be any fixed large enough scale. Consider a much larger scale $N'=N^{P_2}$ with $P_2\gg1$ being an absolute constant.

We want to prove for any $E$ being a generalized eigenvalue of $H_{\epsilon,\omega}(x_0)$, the corresponding extended state $\vec{\psi}$ with $\|\vec{\psi}_0\|=1$
decays  exponentially on $[\frac{N'}{2},N']$. From Lemmata \ref{psiN} and  \ref{paving}, one only needs to show that for any $n\in\mathcal{N}_1=[\sqrt{N'},2N']$, there is some $0\leq j_n\leq \frac{N}{100}$ such that
\begin{equation}\label{state1}
 G_{[1,N]+j_n}(x_0+n\omega;E)\ \mathrm{is}\ good\  \mathrm{ for\  all\ } E.
\end{equation}
Then from (ii) of Proposition \ref{ldt}, the statement (\ref{state1}) is equivalent to
 \begin{equation}\label{state2}
 x_0+n\omega \notin \bigcup_{E}\mathcal{B}_{N}(\omega,E)\  \mathrm{ for\  all\ } n\in[\sqrt{N'},2N'],
\end{equation}
 where $\mathcal{B}_{N}(\omega,E)=\mathcal{B}^{\frac{N}{100}}_{N}(\omega,E)$ (see (\ref{bnm})).  Fortunately, from Lemma \ref{g-}, to prove statement (\ref{state1}), one just needs
 \begin{equation}\label{state3}
 x_0+n\omega \notin \bigcup_{E\in \sigma(H_{[-N_1,N_1]}(x_0))}\mathcal{B}_{N}(\omega,E)\  \mathrm{ for\  all\ } n\in[\sqrt{N'},2N'],
\end{equation}
 where $N_1$ is given by Lemma \ref{g-} and $N_1\ll N'$. Thus we define sets
 $$\mathcal{S}_{N}^{\omega}=\bigcup_{E\in\sigma(H_{[-K,K]}(x_0)),K\leq N^{P_1}}\mathcal{B}_N(\omega,E),$$
 and
 $$\mathcal{S}_{N}=\left\{(\omega,x)\in\mathbb{T}^2: \omega\in \mathrm {DC}_t(N),x\in \mathcal{S}_{N}^{\omega}\right\},$$
 where $\mathrm{DC}_t(N):=\left\{\omega: \|k\omega\|_{\mathbb{T}}\geq \frac{t}{|k|^2}\ \mathrm{for}\ 0<|k|\leq N^2\right\}$. As a result, similarly to the proof of (iii) of Proposition \ref{ldt}, $\mathcal{S}_N$ can be regarded as a semi-algebraic set with sub-exponentially small  Lebsegue measure in $N$ and polynomially bounded  degree in $N$. Using Lemma 9.9 in \cite{BB}, the projective set along the frequencies
 $$\mathcal{R}_{N}=\{\omega:\left(\omega,\{x+n\omega\}\right)\in \mathcal{S}_{N}\  \mathrm{for}\ \mathrm{some}\  \sqrt{N'}\leq n\leq 2N'\}$$
 can also be regarded as a semi-algebraic set with $$\mathrm{Leb}(\mathcal{R}_{N})\leq N^{-C},$$ where $\{x+n\omega\}=x+n\omega\ (\mod 1)$ and $C>1$. The detail elegant analysis  can be found in the proof of Theorem 10.1 in \cite{BB}. Consequently, if $\omega\in\mathrm{DC}_{t}\setminus\mathcal{R}_N$, then $\|\vec{\psi}_j\|\leq e^{-\frac{\rho}{3}j}$ for $\frac{N'}{2}\leq j\leq N'$.

 Note that $\sum\limits_{L\geq N}\mathrm{Leb}(\mathcal{R}_L)< \infty$. Then the set
 $$\mathcal{R}=\bigcap\limits_{K\geq N}\bigcup\limits_{L\geq K}\mathcal{R}_{L}$$ has zero Lebesgue measure by Borel-Cantelli theorem. Thus if $\omega\in\mathrm{DC}_{t}\setminus\mathcal{R}$, then $\|\vec{\psi}_j\|\leq e^{-\frac{\rho}{3}j}$  for $j\geq \frac{N'}{2}$.

 Similarly, we can  show  $\|\vec{\psi}_j\|\leq e^{-\frac{\rho}{3}|j|}$  for $j\leq -N'$ and the proof is finished.
\end{proof}

%\section*{Acknowledgement}
%\section*{Acknowledgements}
%{This work is supported by the National Natural Science Foundation of China (No.11421061).}
 %The revisions of the earlier version of this manuscript were finished when Y. Shi  visited  Prof. S. Jitomirskaya at  University of California, Irvine.  Y. Shi would like to thanks S. Jitomirskya and W. Liu for their kind help. He is  also indebted to Fudan University and Prof. X. Yuan for financial supporting.}

\appendix
\section{}

\begin{lem}[{\L}ojasiewicz inequality, Lemma 7.3 in \cite{BB}]\label{mes}
Let $v$ be a nonconstant analytic function on $\mathbb{T}$. Then there is a constant $\sigma_0=\sigma_0(v)>0$ (depending only on $v$) such that for sufficiently small $\epsilon>0$ and all $E\in\mathbb{R}$,
\begin{equation*}
  \emph{Leb}\{x\in\mathbb{T}:|v(x)-E|<\epsilon\}<\epsilon^{\sigma_0}.
\end{equation*}
\end{lem}

From this lemma, one can obtain the following two useful estimates.

\begin{lem}[Lemma 6.2 in \cite{DKB}]\label{logv}
Let $v$ be as in Lemma \ref{mes}. Then there exists $C=C(v)>0$ (depending only on $v$) such that
$$\int_{\mathbb{T}}\log|v(x)-E|\text{d}x\geq -C.$$
\end{lem}

\begin{lem}\label{1}
Let $v$ be given by Lemma \ref{mes}. Then there is some $0<\sigma_1=\sigma_1(v)<1$  such that
\begin{equation*}
\int_{\mathbb{T}}\log(|v(x)-E|+\epsilon)\text{d}x < \int_{\mathbb{T}}\log|v(x)-E|\text{d}x +\epsilon^{\sigma_1},
\end{equation*}
where $0<\epsilon \ll 1$.
\end{lem}

\begin{proof}
 Let $\sigma_0>0$ be given by Lemma \ref{mes} and $0<\sigma_2<\sigma_0<1$. It is easy to see there is some constant $C_1>0$ such that $\log(1+x)\leq x^{\sigma_2}$ if $x>C_1$. Define $J=\{x\in\mathbb{T}: \frac{\epsilon}{|v(x)-E|}>C_1\}$ and $$J_n=\{x\in J: 2^{-n-1}{C_1^{-1}}{\epsilon}\leq |v(x)-E|<2^{-n}{C_1^{-1}}{\epsilon}\}$$ for $n\in\mathbb{N}$. Then $J=\bigcup\limits_{n=0}^{\infty}J_n$ and $\mathrm{Leb}(J_n)\leq 2^{-n\sigma_0}{C_1^{-\sigma_0}}{\epsilon^{\sigma_0}}$ by Lemma \ref{mes}. Thus

\begin{align*}
       \int_{J}\log \left(1+\frac{\epsilon}{|v(x)-E|}\right)\text{d}x
  \leq &\  \epsilon^{\sigma_{2}}\int_{J}|v(x)-E|^{-\sigma_2}\text{d}x \\
  \leq &\  \sum_{n\geq 0} C_1^{\sigma_2}2^{(n+1)\sigma_2}\mathrm{Leb}(J_n)\\
  \leq &\ \sum_{n\geq 0} C_1^{2\sigma_2-\sigma_0}2^{-(\sigma_0-\sigma_2)n}\epsilon^{\sigma_0} \\
 \leq &\  C_2\epsilon^{\sigma_0},
\end{align*}
and
\begin{align}
 \nonumber\int_{\mathbb{T}\setminus J}\log \left(1+\frac{\epsilon}{|v(x)-E|}\right)\text{d}x
 \nonumber \leq &\ \int_{\{x\in\mathbb{T}:\frac{\epsilon}{C_1}\leq|v(x)-E|<\epsilon^{\sigma_0}\}}\log \left(1+\frac{\epsilon}{|v(x)-E|}\right)\text{d}x \\
               \nonumber &\    +\int_{\{x\in\mathbb{T}:|v(x)-E|\geq\epsilon^{\sigma_0}\}}\log \left(1+\frac{\epsilon}{|v(x)-E|}\right)\text{d}x\\
 \label{ic} \leq &\ \log(1+C_1)\epsilon^{\sigma_0^2}+\epsilon^{1-\sigma_0}.
 \end{align}
In the inequality (\ref{ic}), we use Lemma \ref{mes} and the fact $\log(1+x)\leq x$ for $x\geq 0$. Recalling Lemma \ref{logv} and by letting $\sigma_1=(1-\sigma_0)\sigma_0^2$, one has for $0<\epsilon\ll 1$,
\begin{align*}
\int_{\mathbb{T}}\log(|v(x)-E|+\epsilon)\text{d}x =& \int_{\mathbb{T}}\log|v(x)-E|\text{d}x +\int_{\mathbb{T}}\log\left(1+\frac{\epsilon}{|v(x)-E|}\right)\text{d}x\\
                         \leq & \int_{\mathbb{T}}\log|v(x)-E|\text{d}x + \epsilon^{\sigma_1},
\end{align*}
which completes the proof.
\end{proof}

\begin{rem}
Actually in the proof of Lemma \ref{1}, we have showed that for any $0<\sigma_2<\sigma_0$, $$\int_{\mathbb{T}}|v(x)-E|^{-\sigma_2}\mathrm{d}x>-C(\sigma_2,v)>-\infty, $$  which implies Lemma \ref{logv}.
\end{rem}

We also need the following  Denjoy-Koksma inequality.

\begin{lem}[Lemma 12 in \cite{JA}]\label{dk}
For any continuous real function $u$ on $\mathbb{T}$ and any interval $\mathcal{I}\subset\mathbb{Z}$ with length $N$, we have
\begin{equation*}
 \left |\sum_{j\in \mathcal{I}}u(x+j\omega)- N\int_{\mathbb{T}} u(x)\text{d}x\right|\leq CN^{\frac{1}{2}}\log N,
\end{equation*}
where $\omega\in\mathrm{DC}_t$ and $C$ depends only on $u$ and $t$.
\end{lem}

%\bibliographystyle{abbrv} % abbrv
%\bibliography{references}

% ------------------------------------------------------------------------
\end{document}